\documentclass[12pt]{article}
\usepackage[a4paper,margin=1.1in]{geometry}
\usepackage{amssymb}
\usepackage{amsmath}
\usepackage{amsfonts}
\usepackage{amsthm}
\usepackage{color}
\usepackage{float}
\usepackage[all]{xy}
\usepackage{scalerel}
\usepackage{mathabx}
\usepackage{delimset}

\newcommand{\e}{\mathrm{exp}}
\def\E{\mathop{\mbox{\textup{Exp}}}\nolimits}
\newcommand{\vphiE}{\mathit{Exp}}
\newcommand{\phiE}{\mathit{Exp}^{\prime}}
\def\Sin{\mathop{\mbox{\textup{Sin}}}\nolimits}
\newcommand{\vphiS}{\mathit{Sin}}
\newcommand{\phiS}{\mathit{Sin}^{\prime}}
\newcommand{\sinwt}{\widetilde{\sin}}
\def\Cos{\mathop{\mbox{\textup{Cos}}}\nolimits}
\newcommand{\vphiC}{\mathit{Cos}}
\newcommand{\phiC}{\mathit{Cos}^{\prime}}
\newcommand{\coswt}{\widetilde{\cos}}
\def\Tan{\mathop{\mbox{\textup{Tan}}}\nolimits}
\newcommand{\vphiT}{\mathit{Tan}}
\newcommand{\phiT}{\mathit{Tan}^{\prime}}
\newcommand{\tanwt}{\widetilde{\tan}}
\def\Csc{\mathop{\mbox{\textup{Csc}}}\nolimits}
\newcommand{\vphiCsc}{\mathit{Csc}}
\newcommand{\phiCsc}{\mathit{Csc}^{\prime}}
\newcommand{\cscwt}{\widetilde{\csc}}
\def\Sec{\mathop{\mbox{\textup{Sec}}}\nolimits}
\newcommand{\vphiSec}{\mathit{Sec}}
\newcommand{\phiSec}{\mathit{Sec}^{\prime}}
\newcommand{\secwt}{\widetilde{\sec}}
\def\Cot{\mathop{\mbox{\textup{Cot}}}\nolimits}
\newcommand{\vphiCot}{\mathit{Cot}}
\newcommand{\phiCot}{\mathit{Cot}^{\prime}}
\newcommand{\cotwt}{\widetilde{\cot}}

\def\sech{\mathop{\mbox{\textup{sech}}}\nolimits}
\def\csch{\mathop{\mbox{\textup{csch}}}\nolimits}

\newcommand{\fibonomial}{\genfrac{\{}{\}}{0pt}{}}

\newcommand{\C}{\mathbb{C}}

\newcommand{\N}{\mathbb{N}}

\newcommand{\R}{\mathbb{R}}

\newcommand{\D}{\mathbf{D}}

\newtheorem{theorem}{Theorem}

\newtheorem{definition}{Definition}
\newtheorem{proposition}{Proposition}
\newtheorem{corollary}{Corollary}
\newtheorem{example}{Example}

\begin{document}

\title{\textbf{Lucas-Pantograph Type Exponential, Trigonometric, and Hyperbolic Functions}}
\author{Ronald Orozco L\'opez}
\newcommand{\Addresses}{{
  \bigskip
  \footnotesize

  \textit{E-mail address}, R.~Orozco: \texttt{rj.orozco@uniandes.edu.co}
  
}}

\maketitle
\tableofcontents

\begin{abstract}
In this paper, we include some new results for the Lucas calculus. A Lucas-Pantograph type exponential function is introduced. Additionally, we define Lucas-Pantograph type trigonometric functions, and some of their most notable identities are given: parity, sum and difference formulas, Pythagorean identities, double-angle identities, and some special values. Lucas-Pantograph type hyperbolic functions are also introduced.
\end{abstract}


\section{Introduction}

In this paper, we introduce the Lucas-Pantograph exponential function $\e_{s,t}(x,u)$ defined by
\begin{equation}\label{LP-exp}
    \e_{s,t}(x,u)=\sum_{n=0}^{\infty}u^{\binom{n}{2}}\frac{x^n}{\brk[c]{n}_{s,t}!}
\end{equation}
where $\brk[c]{n}_{s,t}$ are the Lucas sequences. The function (\ref{LP-exp}) is the Lucas analogue of the Pantograph exponential function
\begin{equation}\label{P-exp}
    \e(x,u)=\sum_{n=0}^{\infty}u^{\binom{n}{2}}\frac{x^n}{n!}
\end{equation}
The Pantograph exponential function Eq.(\ref{P-exp}) is solution of the Pantograph functional-differential equation \cite{ise_1,ise_2}
\begin{equation}\label{eqn_sokal}
    f^{\prime}(x)=f(ux),\ \ f(0)=1.
\end{equation}
The Eq.(\ref{eqn_sokal}) is closely related to the generating function for the Tutte polynomials of the complete graph $K_n$ in combinatorics, the Whittaker and Goncharov constant in complex analysis, and the partition function of one-site lattice gas with fugacity $x$ and two-particle Boltzmann weight $q$ in statistical mechanics \cite{sokal}. It would be very interesting to find analogous applications to the above when the parameters $(s,t)$ in the sequence $\brk[c]{n}_{s,t}$ correspond to a calculus defined on the sequences of Fibonacci, Pell, Jacobsthal, Chebysheff, Mersenne, among others.

In this paper, we introduce the Pantograph type Lucas-analog of the binomial theorem
\begin{equation*}
    (x+y)^{n}=\sum_{k=0}^{n}\binom{n}{k}x^ky^{n-k}
\end{equation*}
as the polynomial given by
\begin{equation}\label{eqn_st_bin_theo}
    (x\oplus_{u,v}y)^{(n)}=\sum_{k=0}^{n}\fibonomial{n}{k}_{s,t}u^{\binom{k}{2}}v^{\binom{n-k}{2}}a^kb^{n-k}
\end{equation}
to find an expression for the product 
\begin{equation*}
    \exp_{s,t}(az,u)\exp_{s,t}(bz,v).
\end{equation*}
In addition, we define Lucas-analog of the
Euler's formula 
\begin{equation}\label{eqn_euler}
    e^{ix}=\cos x+i\sin x,
\end{equation}
and of the identity 
\begin{equation}\label{eqn_euler2}
    e^{x+iy}=e^{x}(\cos x+i\sin y).
\end{equation}
An important fact about the functions $\exp_{s,t}(x,u)$ is that if the parameter $u$ takes negative values, then we obtain Lucas-analogues of Eqs. (\ref{eqn_euler}) and (\ref{eqn_euler2}), but without the complex unit $i=\sqrt{-1}$.

Additionally, we define Lucas-Pantograph type trigonometric functions and some of the its most notable identities are given: parity, sum and difference formulas, Pythagorean identities, double-angle identities, and some special values. Lucas-Pantograph type hyperbolic functions are introduced.

In our work, we will use the identities for binomial coefficients:
\begin{align}
    \binom{n+k}{2}&=\binom{n}{2}+\binom{k}{2}+nk,\label{iden7}\\
    \binom{n-k}{2}&=\binom{n}{2}+\binom{k}{2}+k(1-n).\label{iden8}
\end{align}

\section{Calculus on Lucas sequences}

\subsection{Lucas sequences}

The Lucas sequences, depending on the variables $s,t$, are defined by
\begin{align*}
    \brk[c]{0}_{s,t}&=0,\\
    \brk[c]{1}_{s,t}&=1,\\
    \brk[c]{n+2}_{s,t}&=s\{n+1\}_{s,t}+t\{n\}_{s,t}.
\end{align*}
From the Lucas sequence we can obtain the sequences of Fibonacci, Pell, Jacobsthal, Mersenne, $q$-numbers and Chebysheff polynomials of second kind. Another sequence related to Lucas sequences is the sequence $\brk[a]{n}_{s,t}$ defined by
\begin{align*}
    \brk[a]{0}_{s,t}&=2,\\
    \brk[a]{1}_{s,t}&=s,\\
    \brk[a]{n+2}_{s,t}&=s\brk[a]{n+1}_{s,t}+t\brk[a]{n}_{s,t}.
\end{align*}
From the sequence $\brk[a]{n}_{s,t}$ we obtain the numbers of Lucas, Lucas-Pell, Lucas-Jacobsthal, Fermat, and Chebysheff polynomials of first kind. 
The Lucas constant $\varphi$ is the ratio toward which adjacent Lucas sequences tend. This is the only positive root of $x^{2}-sx-t=0$, where
\begin{equation*}
    \varphi=\frac{s+\sqrt{s^{2}+4t}}{2}
\end{equation*}
and its conjugate is
\begin{equation*}
    \phi=s-\varphi=-\frac{t}{\varphi}=\frac{s-\sqrt{s^{2}+4t}}{2}.
\end{equation*}
In the remainder of the paper, we will assume that $s,t\in\C$, $s\neq0$ and $t\neq0$. The Binet's formula for $\brk[c]{n}_{s,t}$ is
\begin{equation}\label{eqn_binet}
    \brk[c]{n}_{s,t}=
    \begin{cases}
    \frac{\varphi^{n}-\phi^{n}}{\varphi-\phi},&\text{ if }s\neq\pm2i\sqrt{t};\\
    n(\pm i\sqrt{t})^{n-1},&\text{ if }s=\pm2i\sqrt{t}.
    \end{cases}
\end{equation}
The Lucasnomial coefficients are defined by
\begin{equation}
    \fibonomial{n}{k}_{s,t}=\frac{\brk[c]{n}_{s,t}!}{\brk[c]{k}_{s,t}!\brk[c]{n-k}_{s,t}!}.\label{eqn_fibo3}
\end{equation}
where $\brk[c]{n}_{s,t}!=\brk[c]{1}_{s,t}\brk[c]{2}_{s,t}\cdots\brk[c]{n}_{s,t}$ is the Lucastorial. The Lucasnomial satisfy the following Pascal recurrence relationships. For $1\leq k\leq n-1$ it is true that
\begin{align}
    \fibonomial{n+1}{k}_{s,t}&=\varphi_{s,t}^{k}\fibonomial{n}{k}_{s,t}+\varphi_{s,t}^{\prime(n+1-k)}\fibonomial{n}{k-1}_{s,t}\label{prop_pascal1},\\
    &=\varphi_{s,t}^{\prime(k)}\fibonomial{n}{k}_{s,t}+\varphi_{s,t}^{n+1-k}\fibonomial{n}{k-1}_{s,t}\label{prop_pascal2}.
\end{align}

\subsection{Lucas-derivative}

\begin{definition}
Set $s,t\in\R$, $s\neq0, t\neq0$. We define the Lucas-derivative $\D_{s,t}$ of the function $f(x)$ as
\begin{equation}
    \D_{s,t}f(x)=\frac{f(\varphi x)-f(\phi x)}{(\varphi-\phi)x}
\end{equation}
for all $x\neq0$ and $(\mathbf{D}_{s,t}f)(0)=f^{\prime}(0)$, provided $f^{\prime}(0)$ exists.
\end{definition}

\begin{example}
For all $s,t\in\R$, $s\neq0,t\neq0$,  and for all $n,k\in\N$,
\begin{equation*}
    \D_{s,t}x^n=\brk[c]{n}_{s,t}x^{n-1}
\end{equation*}
and
\begin{equation}
    \D_{s,t}^kx^n=\frac{\brk[c]{n}_{s,t}!}{\brk[c]{n-k}_{s,t}!}x^{n-k}.
\end{equation}
\end{example}
The following results on the Lucas-derivative are standard.
\begin{itemize}
    \item Linearity. For all functions $f(x)$ and $g(x)$ and for $\alpha\in\C$,
    \begin{align*}
    \mathbf{D}_{s,t}(f(x)+g(x))&=(\mathbf{D}_{s,t}f)(x)+(\mathbf{D}_{s,t}g)(x),\\
    \mathbf{D}_{s,t}(\alpha f(x))&=\alpha(\mathbf{D}_{s,t}f)(x).
\end{align*}
\item Product rule. For all functions $f(x)$ and $g(x)$,
\begin{align}
    \mathbf{D}_{s,t}(f(x)g(x))=f(\varphi x)(\mathbf{D}_{s,t}g)(x)+g(\phi x)(\mathbf{D}_{s,t}f)(x),\label{Leib1}
\end{align}
and
\begin{equation}
    \mathbf{D}_{s,t}(f(x)g(x))=f(\phi x)(\mathbf{D}_{s,t}g)(x)+g(\varphi x)(\mathbf{D}_{s,t}f)(x).\label{Leib2}    
\end{equation}
\item Quotient rule. For all functions $f(x)$ and $g(x)$,
\begin{align}
     \mathbf{D}_{s,t}\left(\frac{f(x)}{g(x)}\right)&=\frac{g(\varphi_{s,t}x)\mathbf{D}_{s,t}f(x)-f(\varphi_{s,t}x)\mathbf{D}_{s,t}g(x)}{g(\varphi_{s,t}x)g(\varphi_{s,t}^{\prime}x)}
\end{align}
and
\begin{align}
     \mathbf{D}_{s,t}\left(\frac{f(x)}{g(x)}\right)&=\frac{g(\varphi_{s,t}^{\prime}x)\mathbf{D}_{s,t}f(x)-f(\varphi_{s,t}^{\prime}x)\mathbf{D}_{s,t}g(x)}{g(\varphi_{s,t}x)g(\varphi_{s,t}^{\prime}x)}.
\end{align}
\end{itemize}

\subsection{Lucas-Integral}

Let $f$ be an arbitrary function and $a$ be a real number. We define the following Lucas-integrals of $f$ as
\begin{align}
    \int_{0}^{a} f(x)d_{s,t}x&=(\varphi-\phi)a\sum_{k=0}^{\infty}\frac{\phi^{k}}{\varphi^{k+1}}f\left(\frac{\phi^{k}}{\varphi^{k+1}}a\right)\text{ if  }\ \Big\vert\frac{\phi}{\varphi}\Big\vert>1,\label{int1}\\
    \int_{0}^{a}f(x)d_{s,t}x&=(\phi-\varphi)a\sum_{k=0}^{\infty}\frac{\varphi^{k}}{\phi^{k+1}}f\left(\frac{\varphi^{k}}{\phi^{k+1}}a\right)\text{ if  }\ \Big\vert\frac{\varphi}{\phi}\Big\vert<1.\label{int2}
\end{align}
If $a$ and $b$ are two nonnegative numbers such that $a<b$, then
\begin{align*}
    \int_{a}^{b}f(x)d_{p,q}x&=\int_{0}^{b}f(x)d_{p,q}x-\int_{0}^{a}f(x)d_{p,q}x.
\end{align*}
Following Eq. (\ref{int1}), the Lucas-integral $f$ is 
\begin{equation}\label{eqn_int_def}
    \int_{a}^{b}f(x)d_{s,t}x=(\varphi-\phi)\sum_{n=0}^{\infty}\Bigg[bf\left(b\frac{\phi^n}{\varphi^{n+1}}\right)-af\left(a\frac{\phi^n}{\varphi^{n+1}}\right)\Bigg]\frac{\phi^n}{\varphi^{n+1}}.
\end{equation}
We will call $f$ a function Lucas-integrable if the series in Eq.(\ref{eqn_int_def}) is convergent. 

{\bf The fundamental theorem of Lucas-calculus}.
If $F(x)$ is an antiderivative of $f(x)$ and $F(x)$ is continuous at $x=0$, then
    \begin{equation}\label{theo_funda}
        \int_{a}^{b}f(x)d_{s,t}x=F(b)-F(a),
    \end{equation}
where $0\leq a<b\leq\infty$.

{\bf The Lucas-integration by parts formula}. 
Suppose $f(x)$ and $g(x)$ are two functions whose ordinary derivatives exist in a neighborhood of $x=0$ and are continuous at $x=0$. Then
\begin{equation}\label{theo_partes2}
    \int_{a}^{b}(\mathbf{D}_{s,t}f)(x)g(\phi x)d_{s,t}x=\big[f\cdot g\big]_{a}^{b}
    -\int_{a}^{b}f(\varphi x)(\mathbf{D}_{s,t}g)(x)d_{s,t}x.
\end{equation}

\section{$\mathbf{u}$-deformed multinomial numbers}

\subsection{$(u,v)$-deformed Lucasnomial theorem}

\begin{definition}\label{def_lucas_theo}
For all nonzero complex numbers $u,v\in\C$, define the $(u,v)$-deformed Lucasnomial theorem, for $x,y$ commuting, as
\begin{equation}
(x\oplus_{u,v}y)_{s,t}^{(n)}=\sum_{k=0}^{n}\fibonomial{n}{k}_{s,t}u^{\binom{n-k}{2}}v^{\binom{k}{2}}x^{n-k}y^{k}\ \ \ \ n\geq0.
\end{equation}
Also, define
\begin{align*}
    (x\oplus_{u,0}y)^{(n)}&=\lim_{v\rightarrow0}(x\oplus_{u,v}y)^{(n)},\ \ \ u\neq0,\\
    (x\oplus_{0,v}y)^{(n)}&=\lim_{u\rightarrow0}(x\oplus_{u,v}y)^{(n)},\ \ \ v\neq0,\\
    (x\oplus_{0,0}y)^{(n)}&=\lim_{u\rightarrow0}\lim_{v\rightarrow0}(x\oplus_{u,v}y)^{(n)}.
\end{align*}
The $(u,v)$-deformed Lucas-analogue of $(x-y)^n$ is defined by
\begin{equation*}
    (x\ominus_{u,v}y)^{(n)}\equiv (x\oplus_{u,v}(-y))^{(n)}.
\end{equation*}
\end{definition}

\begin{example}
\begin{align*}
    (x\oplus_{u,v}y)_{s,t}^{(0)}&=1,\\
    (x\oplus_{u,v}y)_{s,t}^{(1)}&=x+y,\\
    (x\oplus_{u,v}y)_{s,t}^{(2)}&=ux^{2}+\brk[c]{2}_{s,t}xy+vy^{2},\\
    (x\oplus_{u,v}y)_{s,t}^{(3)}&=u^{3}x^{3}+\brk[c]{3}_{s,t}ux^{2}y+\brk[c]{3}_{s,t}vxy^{2}+v^{3}y^{3},\\
    (x\oplus_{u,v}y)_{s,t}^{(4)}&=u^{6}x^{4}+\brk[c]{4}_{s,t}u^{3}x^{3}y+\brk[c]{3}_{s,t}\brk[a]{2}_{s,t}uvx^{2}y^{2}+\brk[c]{4}_{s,t}v^{3}xy^{3}+v^{6}y^{4}.
\end{align*}    
\end{example}
If $u=\varphi$ and $v=\phi$ in Definition \ref{def_lucas_theo}, then we define
\begin{align}
(x\oplus y)_{\varphi,\phi}^{n}\equiv(x\oplus_{\varphi,\phi}y)_{s,t}^{(n)}&=
\begin{cases}
1,& \text{ if }n=0;\\
\prod_{k=0}^{n-1}(\varphi^{k}x+\phi^{k}y),& \text{ if }n\geq1.
\end{cases}\label{eqn_vphi_power}
\end{align}

\begin{theorem}\label{theo_bin_neg}
For all $x,y,u,v\in\C$,
    \begin{align}
        (x\oplus_{-u,-v}y)_{s,t}^{(2n)}&=(-1)^{n}(x\ominus_{u,v}y)_{s,t}^{(2n)}.\\
        (x\oplus_{-u,-v}y)_{s,t}^{(2n+1)}&=(-1)^{n}(x\oplus_{u,v}y)_{s,t}^{(n)}.
    \end{align}
\end{theorem}
\begin{proof}
We use the Definition \ref{def_lucas_theo} and the Eq.(\ref{iden8})
    \begin{align*}
        (x\oplus_{-u,-v}y)_{s,t}^{(n)}&=\sum_{k=0}^{n}\fibonomial{n}{k}_{s,t}(-u)^{\binom{n-k}{2}}(-v)^{\binom{k}{2}}x^{n-k}y^{k}\\
        &=\sum_{k=0}^{n}\fibonomial{n}{k}_{s,t}u^{\binom{n-k}{2}}v^{\binom{k}{2}}(-1)^{\binom{n-k}{2}+\binom{k}{2}}x^{n-k}y^{k}\\
        &=(-1)^{\binom{n}{2}}\sum_{k=0}^{n}\fibonomial{n}{k}_{s,t}u^{\binom{n-k}{2}}v^{\binom{k}{2}}x^{n-k}(-1)^{k(1-n)}y^{k}\\
        &=(-1)^{\binom{n}{2}}(x\oplus_{u,v}(-1)^{1-n}y)_{s,t}^{(n)}.
    \end{align*}
Finally, we use the parity of $n$.
\end{proof}

\begin{theorem}\label{theo_propi_uvFbinom}
For all $x,y,z,u,v\in\C$ the $(u,v)$-deformed Lucasnomial formula has the following properties
\begin{enumerate}
    \item $(x\oplus_{u,v}y)_{s,t}^{(n+1)}=x(ux\oplus_{u,v}\varphi y)_{s,t}^{(n)}+y(\phi x\oplus_{u,v}vy)_{s,t}^{(n)}$.
    \item $(x\oplus_{u,v}y)_{s,t}^{(n+1)}=x(ux\oplus_{u,v}\phi y)_{s,t}^{(n)}+y(\varphi x\oplus_{u,v}vy)_{s,t}^{(n)}$.
    \item $(x\oplus_{au,av}y)^{(n)}=a^{\binom{n}{2}}(x\oplus_{u,v}y)^{(n)}$.
    \item $(x\oplus_{u,v}y)^{(n)}=(y\oplus_{v,u}x)^{(n)}$,
    \item $z^n(x\oplus_{u,v}y)^{(n)}=(zx\oplus_{u,v}zy)^{(n)}$.
    \item $(x\oplus_{u,v}0)^{(\alpha)}=u^{\binom{\alpha}{2}}x^\alpha$, for all $x\in\C$.
    \item $(0\oplus_{u,v}y)^{(\alpha)}=v^{\binom{\alpha}{2}}y^\alpha$, for all $y\in\C$.
\end{enumerate}
\end{theorem}
\begin{proof}
By Definition \ref{def_lucas_theo},
\begin{align}
    (x\oplus_{u,v}y)_{s,t}^{(n+1)}&=\sum_{k=0}^{n}\fibonomial{n+1}{k}_{s,t}u^{\binom{n+1-k}{2}}v^{\binom{k}{2}}x^{n+1-k}y^{k}.
\end{align}
By extracting the first summand and by using Eq.(\ref{prop_pascal1}), we obtain
\begin{align*}
    &(x\oplus_{u,v}y)_{s,t}^{(n+1)}\\
    &\hspace{1cm}=u^{\binom{n+1}{2}}x^{n+1}
    +\sum_{k=1}^{n-1}\left(\varphi^{k}\fibonomial{n}{k}_{s,t}+\phi^{n-k+1}\fibonomial{n}{k-1}_{s,t}\right)v^{\binom{k}{2}}u^{\binom{n+1-k}{2}}y^{k}x^{n+1-k}\\
    &\hspace{1cm}=u^{\binom{n+1}{2}}x^{n+1}
    +x\sum_{k=1}^{n-1}\fibonomial{n}{k}_{s,t}v^{\binom{k}{2}}u^{\binom{n-k}{2}}(\varphi y)^{k}(u x)^{n-k}\\
    &\hspace{4cm}+\sum_{k=1}^{n-1}\phi^{n-k+1}\fibonomial{n}{k-1}_{s,t}v^{\binom{k}{2}}u^{\binom{n+1-k}{2}}y^{k}x^{n+1-k}.
\end{align*}
In addition, we rearrange the second summation for $k$
\begin{align*}
    &(x\oplus_{u,v}y)_{s,t}^{(n+1)}\\
    &\hspace{1cm}=u^{\binom{n+1}{2}}x^{n+1}
    +x\sum_{k=1}^{n-1}\fibonomial{n}{k}_{s,t}v^{\binom{k}{2}}u^{\binom{n-k}{2}}(\varphi y)^{k}(u x)^{n-k}\\
    &\hspace{2cm}+y\sum_{k=0}^{n}\fibonomial{n}{k}_{s,t}v^{\binom{k}{2}}u^{\binom{n-k}{2}}(v y)^{k}(\phi x)^{n-k}\\
    &\hspace{1cm}=x\sum_{k=0}^{n}\fibonomial{n}{k}_{s,t}v^{\binom{k}{2}}u^{\binom{n-k}{2}}(\varphi y)^{k}(u x)^{n-k}
    +y\sum_{k=0}^{n}\fibonomial{n}{k}_{s,t}v^{\binom{k}{2}}u^{\binom{n-k}{2}}(vy)^{k}(\phi x)^{n-k}\\
    &\hspace{1cm}=x(ux\oplus_{u,v}\varphi_{s,t}y)_{s,t}^{(n)}+y(\varphi_{s,t}^{\prime}x\oplus_{u,v}vy)_{s,t}^{(n)}
\end{align*}
and thus we obtain the first statement. By using Lucas-Pascal formula, Eq.(\ref{prop_pascal2}), we obtain 2. The proof of $3$, $4$, $5$, $6$, and $7$ are trivials.
\end{proof}

\begin{theorem}\label{theo_der_bino}
For all $n\in\N$
\begin{enumerate}
    \item $\mathbf{D}_{s,t}(x\oplus_{u,v}a)_{s,t}^{(n)}=\brk[c]{n}_{s,t}(ux\oplus_{u,v}a)_{s,t}^{(n-1)}$.
    \item $\mathbf{D}_{s,t}(a\oplus_{u,v}x)_{s,t}^{(n)}=\brk[c]{n}_{s,t}(a\oplus_{u,v}vx)_{s,t}^{(n-1)}$.
    \item $\mathbf{D}_{s,t}(a\ominus_{u,v}x)_{s,t}^{(n)}=-\brk[c]{n}_{s,t}(a\ominus_{u,v}vx)_{s,t}^{(n-1)}$.
\end{enumerate}
\end{theorem}
\begin{proof}
We prove assertion 1. The other assertions are proved similarly.
\begin{align*}
    \mathbf{D}_{s,t}(x\oplus_{u,v}a)_{s,t}^{(n)}&=\mathbf{D}_{s,t}\left(\sum_{k=0}^{n}\fibonomial{n}{k}_{s,t}u^{\binom{k}{2}}v^{\binom{n-k}{2}}x^{k}a^{n-k}\right)\\
    &=\sum_{k=1}^{n}\fibonomial{n}{k}_{s,t}u^{\binom{k}{2}}v^{\binom{n-k}{2}}\brk[c]{k}_{s,t}x^{k-1}a^{n-k}\\
    &=\sum_{k=0}^{n-1}\fibonomial{n}{k+1}_{s,t}u^{\binom{k}{2}}u^{k}v^{\binom{n-k-1}{2}}\brk[c]{k+1}_{s,t}x^{k}a^{n-k-1}\\
    &=\sum_{k=0}^{n-1}\frac{\brk[c]{n}_{s,t}\brk[c]{n-1}_{s,t}\cdots\brk[c]{n-k}_{s,t}}{\brk[c]{k}_{s,t}!}u^{\binom{k}{2}}v^{\binom{n-1-k}{2}}(u x)^{k}a^{n-1-k}\\
    &=\brk[c]{n}_{s,t}\sum_{k=0}^{n-1}\frac{\brk[c]{n-1}_{s,t}\cdots\brk[c]{n-1-k+1}_{s,t}}{\brk[c]{k}_{s,t}!}u^{\binom{k}{2}}v^{\binom{n-1-k}{2}}(u x)^{k}a^{n-1-k}\\
    &=\brk[c]{n}_{s,t}(ux\oplus_{u,v}a)^{(n-1)}.
\end{align*}    
\end{proof}


\subsection{$(u_{1},u_{2},\ldots,u_{k})$-deformed Lucasnomiales numbers}

\begin{definition}
We define the $(u,v)$-deformed Lucasnomial zero as
\begin{equation}
    \overline{0}_{u,v}^{(n)}=\sum_{k=0}^{n}(-1)^k\fibonomial{n}{k}_{s,t}u^{\binom{n-k}{2}}v^{\binom{k}{2}}.
\end{equation}
\end{definition}

\begin{definition}
Let $u_{1},u_{2},\ldots,u_{k}$ be non-zero complex numbers. Define inductively the following $\mathbf{u}$-deformed Lucasnomiales numbers as
\begin{align}
\overline{1}_{(u_{1},u_{2})}^{(n)}&=(1\oplus_{u_{1},u_{2}}0)_{s,t}^{(n)}=(0\oplus_{u_{1},u_{2}}1)_{s,t}^{(n)}=1,\\
\overline{2}_{(u_{1},u_{2})}^{(n)}&=(1\oplus_{u_{1},u_{2}}1)_{s,t}^{(n)}=\sum_{k=0}^{n}\fibonomial{n}{k}_{s,t}u_{1}^{\binom{n-k}{2}}u_{2}^{\binom{k}{2}},\\
\overline{3}_{(u_{1},u_{2},u_{3})}^{(n)}&=(\overline{2}_{(u_{1},u_{2})}\oplus_{1,u_{3}}1)_{s,t}^{(n)}=(1\oplus_{u_{1},1}\overline{2}_{(u_{2},u_{3})})_{s,t}^{(n)}\nonumber\\
&=\sum_{k_{1}+k_{2}+k_{3}=n}\frac{\brk[c]{n}_{s,t}!}{\brk[c]{k_{1}}_{s,t}!\brk[c]{k_{2}}_{s,t}!\brk[c]{k_{3}}_{s,t}!}u_{1}^{\binom{k_{1}}{2}}u_{2}^{\binom{k_{2}}{2}}u_{3}^{\binom{k_{3}}{2}}
\end{align}
and
\begin{align}
    \overline{m+1}_{(u_{1},\ldots,u_{m+1})}^{(n)}&=(\overline{m}_{(u_{1},\ldots,u_{m})}\oplus_{1,u_{m+1}}1)_{s,t}^{(n)}\nonumber\\
    &=\sum_{k_{1}+\cdots+k_{m+1}=n}\frac{\brk[c]{n}_{s,t}!}{\brk[c]{k_{1}}_{s,t}!\cdots\brk[c]{k_{m+1}}_{s,t}!}u_{1}^{\binom{k_{1}}{2}}\cdots u_{m+1}^{\binom{k_{m+1}}{2}}
\end{align}
for all integer $n\geq1$. For $n=0$, $\overline{m}_{u_{1},\ldots,u_{m}}^{(0)}=1$. If $u_{1}=u_{2}=\cdots=u_{k}=\cdots=u$, then set
\begin{equation}
    \overline{m}_{u_{1},\ldots,u_{m}}^{(n)}=\overline{m}_{\mathbf{u}}^{(n)}.
\end{equation}
\end{definition}

\section{Lucas-Pantograph type exponential functions}

\begin{definition}
Set $u\in\C$. The Lucas-Pantograph type exponential function is defined to be
\begin{equation}\label{def_uexp}
    \e_{s,t}(z,u)=
    \begin{cases}
        \sum_{n=0}^{\infty}u^{\binom{n}{2}}\frac{z^n}{\brk[c]{n}_{s,t}!},&\text{ if }u\neq0;\\
        1+z,&\text{ if }u=0.
    \end{cases}
\end{equation}
Some exponential-Lucas functions are defined from Eq.(\ref{def_uexp}). If $u=1$,
    \begin{equation}
        \e_{s,t}(z)\equiv\e_{s,t}(z,1)=\sum_{n=0}^{\infty}\frac{z^n}{\brk[c]{n}_{s,t}!}.
    \end{equation}
If $u=\varphi\phi=-t$,
    \begin{equation}
        \E_{s,t}(z)\equiv\e_{s,t}(z,-t)=\sum_{n=0}^{\infty}(-t)^{\binom{n}{2}}\frac{z^n}{\brk[c]{n}_{s,t}!}.
    \end{equation}
If $u=\varphi$,
    \begin{equation}
        \vphiE_{s,t}(z)\equiv\e_{s,t}(z,\varphi)=        \sum_{n=0}^{\infty}\varphi^{\binom{n}{2}}\frac{z^n}{\brk[c]{n}_{s,t}!}.
    \end{equation}
If $u=\phi$,
    \begin{equation}
        \phiE_{s,t}(z)\equiv\e_{s,t}(z,\phi)=        \sum_{n=0}^{\infty}\phi^{\binom{n}{2}}\frac{z^n}{\brk[c]{n}_{s,t}!}.
    \end{equation}    
\end{definition}
The function $\e_{s,t}(x.u)$ is solution of the equation $\mathbf{D}_{s,t}y=y(ux)$. The Lucas-antiderivative of $\exp_{s,t}(x,u)$ is
\begin{equation}
    \int\exp_{s,t}(x,u)d_{s,t}x=u\exp_{s,t}(x/u,u).
\end{equation}

\begin{theorem}
For all $k\in\N$
    \begin{equation}
        \D_{s,t}^k\e_{s,t}(az,u)=a^ku^{\binom{k}{2}}\e_{s,t}(au^kz,u).
    \end{equation}
\end{theorem}
\begin{proof}
    \begin{align*}
        \D^k\e_{s,t}(az,u)&=\sum_{n=0}^{\infty}u^{\binom{n}{2}}\frac{a^n\D_{s,t}^kz^n}{\brk[c]{n}_{s,t}!}=\sum_{n=k}^{\infty}u^{\binom{n}{2}}\frac{a^nz^{n-k}}{\brk[c]{n-k}_{s,t}!}
        =a^k\sum_{n=0}^{\infty}u^{\binom{n+k}{2}}\frac{(az)^{n}}{\brk[c]{n}_{s,t}!}\\
        &=a^ku^{\binom{k}{2}}\e_{s,t}(au^kz,u).
    \end{align*}
\end{proof}

\begin{definition}\label{def_expbin}
For all $x,y,u,v\in\C$,
    \begin{align}
        \e_{s,t}(x\oplus_{u,v}y)&=\sum_{n=0}^{\infty}\frac{(x\oplus_{u,v}y)_{s,t}^{(n)}}{\brk[c]{n}_{s,t}!}=\e_{s,t}(x,u)\e_{s,t}(y,v).
    \end{align}
If $u=\varphi$ and $v=\phi$, then
\begin{equation}
    \e_{s,t}(x\oplus y)=\sum_{n=0}^{\infty}\frac{(x\oplus y)_{\varphi,\phi}^{(n)}}{\brk[c]{n}_{s,t}!}=\vphiE_{s,t}(x)\phiE_{s,t}(y).
\end{equation}
\end{definition}

\begin{proposition}
For all $x,u,v\in\C$,
    \begin{equation}
        \e_{s,t}(-x,u)=\frac{\e_{s,t}(\overline{0}_{u,v}x)}{\e_{s,t}(x,v)}
    \end{equation}
and
\begin{equation}
        \vphiE_{s,t}(z)\phiE_{s,t}(-z)=1.
    \end{equation}
\end{proposition}

\begin{theorem}
For all $a,c,x\in\C$
    \begin{enumerate}
        \item $\D_{s,t}\exp_{s,t}(ax\oplus_{u,v}c)=a\exp_{s,t}(aux\oplus_{u,v}c)$.
        \item $\int\e_{s,t}(ax\oplus_{u,v}c)d_{s,t}x=\frac{u}{a}\e_{s,t}(\frac{a}{u}x\oplus_{u,v}c)$.
        \item $\D_{s,t}\exp_{s,t}(a\oplus_{u,v}cx)=c\exp_{s,t}(a\oplus_{u,v}cvx)$.
        \item $\int\e_{s,t}(a\oplus_{u,v}cx)d_{s,t}x=\frac{v}{c}\e_{s,t}(a\oplus_{u,v}\frac{c}{v}x)$.
    \end{enumerate}
\end{theorem}

\begin{theorem}\label{prop_exp_prod}
Take $\alpha,\beta,u,v\in\C$. The function $\e_{s,t}((\alpha\oplus_{u,v}\beta)x)$ fulfill the identities
    \begin{enumerate}
        \item
        \begin{equation}
            \mathbf{D}_{s,t}\{\e_{s,t}((\alpha\oplus_{u,v}\beta)x)\}=\alpha\e_{s,t}((\alpha u\oplus_{u,v}\beta\varphi)x)+\beta\e_{s,t}((\alpha\phi\oplus_{u,v}\beta v)x).
        \end{equation}
        If $u=\varphi$ and $v=\phi$, then
        \begin{equation}
            \mathbf{D}_{s,t}\{\e_{s,t}((\alpha\oplus\beta)x)\}=\alpha\e_{s,t}((\alpha \oplus\beta)\varphi x)+\beta\e_{s,t}((\alpha\oplus\beta )\phi x).
        \end{equation}
        That is, the function $\e_{s,t}((\alpha\oplus\beta)x)$ is a solution of the proportional functional equation
        \begin{equation}
            \D_{s,t}f(x)=\alpha f(\varphi x)+\beta f(\phi x).
        \end{equation}
        \item 
        \begin{multline}
            \int\e_{s,t}((\alpha\oplus_{u,v}\beta)x)d_{s,t}x=\frac{u}{\alpha}\e_{s,t}\left(\frac{\alpha}{u}
    x,u\right)\e_{s,t}\left(\frac{\beta}{\phi}x,v\right)\\
    -\frac{u\beta}{\alpha\phi}\int\e_{s,t}\left(\frac{\varphi\alpha}{u}x,u\right)\e_{s,t}\left(\frac{v\beta}{\phi}x,v\right)d_{s,t}x.
        \end{multline}
        If $u=\varphi$ and $v=\phi$ and $\alpha\neq\varphi$, $\beta\neq-\phi$, then
        \begin{equation}
            \int\e_{s,t}((\alpha\oplus_{u,v}\beta)x)d_{s,t}x=\frac{\alpha\phi}{\alpha\phi+\beta\varphi}\e_{s,t}\left(\left(\frac{\alpha}{\varphi}\oplus\frac{\beta}{\phi}\right)x\right).
        \end{equation}
    \end{enumerate}
\end{theorem}
\begin{proof}
From Eq.(\ref{eqn_vphi_power}),
\begin{align*}
    \D_{s,t}\exp_{s,t}((\alpha\oplus_{u,v}\beta)x)    &=\D_{s,t}\left(\sum_{n=0}^{\infty}(\alpha\oplus_{u,v}\beta)_{s,t}^{(n)}\frac{x^n}{\brk[c]{n}_{s,t}!}\right)\\
    &=\sum_{n=0}^{\infty}(\alpha\oplus_{u,v}\beta)_{s,t}^{(n+1)}\frac{x^n}{\brk[c]{n}_{s,t}!}\\
    &=\alpha\sum_{n=0}^{\infty}(\alpha u\oplus_{u,v}\beta\varphi)^{n}_{s,t}\frac{x^n}{\brk[c]{n}_{s,t}!}+\beta\sum_{n=0}^{\infty}(\alpha \phi\oplus_{u,v}\beta v)_{s,t}^{n}\frac{x^n}{\brk[c]{n}_{s,t}!}\\
    &=\alpha\exp_{s,t}((\alpha u\oplus_{u,v}\beta\varphi)x)+\beta\exp_{s,t}((\alpha \phi\oplus_{u,v}\beta v)x).
\end{align*}
From Lucas-integration by parts, formula Eq.(\ref{theo_partes2}),
\begin{align*}
    \int\e_{s,t}((\alpha\oplus_{u,v}\beta)x)d_{s,t}x
    &=\int\e_{s,t}(\alpha x,u)\e_{s,t}(\beta x,v)d_{s,t}x\\
    &=\frac{u}{\alpha}\int\e_{s,t}\left(\phi\left(\frac{\beta}{\phi}
    \right)x,v\right)\D_{s,t}\e_{s,t}\left(\left(\frac{\alpha}{u}\right)x,u\right)d_{s,t}x\\
    &=\frac{u}{\alpha}\bigg[\e_{s,t}\left(\frac{\alpha}{u}
    x,u\right)\e_{s,t}\left(\frac{\beta}{\phi}x,v\right)\\
    &\hspace{2cm}-\int\e_{s,t}\left(\frac{\varphi\alpha}{u} x,u\right)\D_{s,t}\e_{s,t}\left(\frac{\beta}{\phi}
    x,v\right)d_{s,t}x\bigg]\\
    &=\frac{u}{\alpha}\e_{s,t}\left(\frac{\alpha}{u}
    x,u\right)\e_{s,t}\left(\frac{\beta}{\phi}x,v\right)\\
    &\hspace{2cm}-\frac{u\beta}{\alpha\phi}\int\e_{s,t}\left(\frac{\varphi\alpha}{u}x,u\right)\e_{s,t}\left(\frac{v\beta}{\phi}x,v\right)d_{s,t}x.
\end{align*} 
Now set $u=\varphi$ and $v=\phi$, then
\begin{align*}
    \int\e_{s,t}((\alpha\oplus\beta)x)d_{s,t}&=\frac{\varphi}{\alpha}\e_{s,t}\left(\frac{\alpha}{\varphi}
    x,\varphi\right)\e_{s,t}\left(\frac{\beta}{\phi}x,\phi\right)\\
    &\hspace{1cm}-\frac{\varphi\beta}{\alpha\phi}\int\e_{s,t}\left(\alpha x,\varphi\right)\e_{s,t}\left(\beta x,\phi\right)d_{s,t}x\\
    &=\frac{\varphi}{\alpha}\e_{s,t}\left(\left(\frac{\alpha}{\varphi}\oplus\frac{\beta}{\phi}\right)x\right)
    -\frac{\varphi\beta}{\alpha\phi}\int\e_{s,t}((\alpha\oplus\beta)x)d_{s,t}x
\end{align*}
and
\begin{align*}
    \int\e_{s,t}((\alpha\oplus\beta)x)d_{s,t}
    =\frac{\alpha\phi}{\alpha\phi+\beta\varphi}\e_{s,t}\left(\left(\frac{\alpha}{\varphi}\oplus\frac{\beta}{\phi}\right)x\right).
\end{align*}
The proof is completed.
\end{proof}

\begin{definition}
Let $u_{1},\ldots,u_{k}$ be non-zero complex numbers. Define the $\mathbf{u}$-multinomial exponential-Lucas function as
\begin{equation}
    \e_{s,t}(\overline{k}_{\mathbf{u}}x)=\sum_{n=0}^{\infty}\overline{k}_{\mathbf{u}}^{(n)}\frac{x^n}{\brk[c]{n}_{s,t}!}.
\end{equation}
\end{definition}

Now the following proposition follows easily.
\begin{proposition}\label{prop_prod_nexp}
    \begin{equation}
        \prod_{k=1}^{n}\e_{s,t}(x,u_{k})=\e_{s,t}(\overline{n}_{\mathbf{u}}x).
    \end{equation}
\end{proposition}

\section{Lucas-Pantograph type trigonometric functions}

\subsection{Definition and analytic properties}

\begin{definition}
For all $s,t,u\in\C$, $s\neq0,t\neq0$, the Lucas-Pantograph type sine function is defined as
\begin{equation*}
    \sin_{s,t}(z,u)=
    \begin{cases}
        \sum_{n=0}^{\infty}(-1)^{n}u^{\binom{2n+1}{2}}\frac{z^{2n+1}}{\brk[c]{2n+1}_{s,t}!},&\text{ if }u\neq0;\\
        z,&\text{ if }u=0.
    \end{cases}
\end{equation*}
In addition, we define
\begin{align*}
    \sin_{s,t}(z)&=\sin_{s,t}(z,1)=\sum_{n=0}^{\infty}(-1)^{n}\frac{z^{2n+1}}{\brk[c]{2n+1}_{s,t}!},\\
    \Sin_{s,t}(z)&=\sin_{s,t}(z,-t)=\sum_{n=0}^{\infty}(-1)^{n}(-t)^{\binom{2n+1}{2}}\frac{z^{2n+1}}{\brk[c]{2n+1}_{s,t}!},\\
    \vphiS_{s,t}(z)&=\sin_{s,t}(z,\varphi)=\sum_{n=0}^{\infty}(-1)^{n}\varphi^{\binom{2n+1}{2}}\frac{z^{2n+1}}{\brk[c]{2n+1}_{s,t}!},\\
    \phiS_{s,t}(z)&=\sin_{s,t}(z,\phi)=\sum_{n=0}^{\infty}(-1)^{n}\phi^{\binom{2n+1}{2}}\frac{z^{2n+1}}{\brk[c]{2n+1}_{s,t}!}.
\end{align*}
\end{definition}

\begin{definition}
For all $s,t,u\in\C$, $s\neq0,t\neq0$, the Lucas-Pantograph type cosine function is defined as
\begin{equation*}
    \cos_{s,t}(z,u)=
    \begin{cases}
        \sum_{n=0}^{\infty}(-1)^{n}u^{\binom{2n}{2}}\frac{z^{2n}}{\brk[c]{2n}_{s,t}!},&\text{ if }u\neq0;\\
        1,&\text{ if }u=0.
    \end{cases}
\end{equation*}
In addition, define
\begin{align*}
    \cos_{s,t}(z)&=\cos_{s,t}(x,1)=\sum_{n=0}^{\infty}(-1)^{n}\frac{z^{2n}}{\brk[c]{2n}_{s,t}!},\\
    \Cos_{s,t}(z)&=\cos_{s,t}(x,1)=\sum_{n=0}^{\infty}(-1)^{n}(-t)^{\binom{2n}{2}}\frac{z^{2n}}{\brk[c]{2n}_{s,t}!},\\
    \vphiC_{s,t}(z)&=\cos_{s,t}(z,\varphi)=\sum_{n=0}^{\infty}(-1)^{n}\varphi^{\binom{2n}{2}}\frac{z^{2n}}{\brk[c]{2n}_{s,t}!},\\
    \phiC_{s,t}(z)&=\cos_{s,t}(z,\phi)=\sum_{n=0}^{\infty}(-1)^{n}\phi^{\binom{2n}{2}}\frac{z^{2n}}{\brk[c]{2n}_{s,t}!}.
\end{align*}
\end{definition}

Exton \cite{exton} also defined the following trigonometric functions
\begin{align}
    \sin_{1,q}(x,\sqrt{q})&=\sum_{n=0}^{\infty}q^{\frac{1}{2}\binom{n}{2}}\frac{x^{2n+1}}{\brk[s]{2n+1}_q!}\label{eqn_sin_ext},\\
    \cos_{1,q}(x,\sqrt{q})&=\sum_{n=0}^{\infty}q^{\frac{1}{2}\binom{n}{2}}\frac{x^{2n}}{\brk[s]{2n}_q!}\label{eqn_cos_ext}.
\end{align}
In addition, Fitouhi et al. \cite{fitouhi_1,fitouhi_2} introduced the following two trigonometric functions
\begin{align}
    \sin_{1,q}(x,q^2)&=\sum_{n=0}^{\infty}q^{2\binom{n}{2}}\frac{x^{2n+1}}{\brk[s]{2n+1}_q!}\label{eqn_sin_fit},\\
    \cos_{1,q}(x,q^2)&=\sum_{n=0}^{\infty}q^{2\binom{n}{2}}\frac{x^{2n}}{\brk[s]{2n}_q!}\label{eqn_cos_fit}.
\end{align}

  

\begin{definition}
For all $s,t,u\in\C$, $s\neq0,t\neq0$, the Lucas-Pantograph type tangent function is defined as
\begin{equation*}
    \tan_{s,t}(z,u)=
    \begin{cases}
        \frac{\sin_{s,t}(z,u)}{\cos_{s,t}(z,u)},&\text{ if }u\neq0;\\
        z,&\text{ if }u=0.
    \end{cases}
\end{equation*}
In addition, define
\begin{align*}
    \tan_{s,t}(z)&=\frac{\sin_{s,t}(z)}{\cos_{s,t}(z)},\hspace{0.3cm} \Tan_{s,t}(z)=\frac{\Sin_{s,t}(z)}{\Cos_{s,t}(z)},\\
    \vphiT_{s,t}(z)&=\frac{\vphiS_{s,t}(z)}{\vphiC_{s,t}(z)},\hspace{0.2cm}\phiT_{s,t}(z)=\frac{\phiS_{s,t}(z)}{\phiC_{s,t}(z)}.
\end{align*}
\end{definition}

\begin{definition}
For all $s,t,u\in\C$, $s\neq0,t\neq0$, the Lucas-Pantograph type cotangent function is defined as
\begin{equation*}
    \cot_{s,t}(z,u)=
    \begin{cases}
        \frac{\cos_{s,t}(z,u)}{\sin_{s,t}(z,u)},&\text{ if }u\neq0;\\
        \frac{1}{z},&\text{ if }u=0.
    \end{cases}
\end{equation*}
In addition, define
\begin{align*}
    \cot_{s,t}(z)&=\frac{\cos_{s,t}(z)}{\sin_{s,t}(z)},\hspace{0.3cm} \Cot_{s,t}(z)=\frac{\Cos_{s,t}(z)}{\Sin_{s,t}(z)},\\
    \vphiCot_{s,t}(z)&=\frac{\vphiC_{s,t}(z)}{\vphiS_{s,t}(z)},\hspace{0.2cm}\phiCot_{s,t}(z)=\frac{\phiC_{s,t}(z)}{\phiS_{s,t}(z)}.
\end{align*}
\end{definition}

\begin{definition}
For all $s,t,u\in\C$, $s\neq0,t\neq0$, the Lucas-Pantograph type secant function is defined as
\begin{equation*}
    \sec_{s,t}(z,u)=
    \begin{cases}
        \frac{1}{\cos_{s,t}(z,u)},&\text{ if }u\neq0;\\
        1,&\text{ if }u=0.
    \end{cases}
\end{equation*}
In addition, define
\begin{align*}
    \sec_{s,t}(z)&=\frac{1}{\cos_{s,t}(z)},\hspace{0.3cm} \Sec_{s,t}(z)=\frac{1}{\Cos_{s,t}(z)},\\
    \vphiSec_{s,t}(z)&=\frac{1}{\vphiC_{s,t}(z)},\hspace{0.2cm}\phiSec_{s,t}(z)=\frac{1}{\phiC_{s,t}(z)}.
\end{align*}
\end{definition}

\begin{definition}
For all $s,t,u\in\C$, $s\neq0,t\neq0$, the Lucas-Pantograph type cosecant function is defined as
\begin{equation*}
    \csc_{s,t}(z,u)=
    \begin{cases}
        \frac{1}{\sin_{s,t}(z,u)},&\text{ if }u\neq0;\\
        \frac{1}{z},&\text{ if }u=0.
    \end{cases}
\end{equation*}
In addition, define
\begin{align*}
    \csc_{s,t}(z)&=\frac{1}{\sin_{s,t}(z)},\hspace{0.3cm} \Csc_{s,t}(z)=\frac{1}{\Sin_{s,t}(z)},\\
    \vphiCsc_{s,t}(z)&=\frac{1}{\vphiS_{s,t}(z)},\hspace{0.2cm}\phiCsc_{s,t}(z)=\frac{1}{\phiS_{s,t}(z)}.
\end{align*}
\end{definition}

\begin{theorem}\label{theo_der_trigo}
The Lucas derivatives of the Lucas-Pantograph type trigonometric functions are
\begin{enumerate}
    \item $\mathbf{D}_{s,t}(\sin_{s,t}(x,u))=\cos_{s,t}(ux,u)$.
    \item $\mathbf{D}_{s,t}(\cos_{s,t}(x,u))=-\sin_{s,t}(ux,u)$.
    \item $\mathbf{D}_{s,t}(\tan_{s,t}(x,u))=\frac{\cos_{s,t}(ux,u)}{\cos_{s,t}(\varphi x,u)}+\tan_{s,t}(\phi x,u)\frac{\sin_{s,t}(ux,u)}{\cos_{s,t}(\varphi x,u)}$.
    \item $\mathbf{D}_{s,t}(\cot_{s,t}(x,u))=-\frac{\sin_{s,t}(ux,u)}{\sin_{s,t}(\varphi x,u)}-\cot_{s,t}(\phi x,u)\frac{\cos_{s,t}(ux,u)}{\sin_{s,t}(\varphi x,u)}$.
    \item $\mathbf{D}_{s,t}(\sec_{s,t}(x,u))=\frac{\sin_{s,t}(ux,u)}{\cos_{s,t}(\varphi x,u)\cos_{s,t}(\phi x,u)}$.
    \item $\mathbf{D}_{s,t}(\csc_{s,t}(x,u))=-\frac{\cos_{s,t}(ux,u)}{\sin_{s,t}(\varphi x,u)\sin_{s,t}(\phi x,u)}$.
\end{enumerate}
\end{theorem}

\begin{corollary}
For all $u\in\C$
    \begin{enumerate}
        \item $\D_{s,t}^2(\sin_{s,t}(x,u))=-u\sin_{s,t}(u^2x,u)$.
        \item $\D_{s,t}^2(\cos_{s,t}(x,u))=-u\cos_{s,t}(u^2x,u)$.
    \end{enumerate}
\end{corollary}
Then the functions $\sin_{s,t}(x,u)$ and $\cos_{s,t}(x,u)$ are solutions of the proportional functional equation 
\begin{equation}
    \D_{s,t}^2f(x)=-uf(u^2x).
\end{equation}

\subsection{Lucas-Pantograph-Euler formula}

Next, we obtain the deformed Lucas-Pantograph-analog of Euler's formula 
\begin{equation*}
e^{ix}=\cos x+i\sin x.   
\end{equation*}

\begin{theorem}\label{prop_exp_complex_neg}
For all $u\in\C$
\begin{enumerate}
    \item $\exp_{s,t}(ix,u)=\cos_{s,t}(x,u)+i\sin_{s,t}(x,u)$.
    \item $\exp_{s,t}(z,-u)=\cos_{s,t}(z,u)+\sin_{s,t}(z,u)$.
\end{enumerate}
\end{theorem}
\begin{proof}
When $u=0$, then 
\begin{align*}
    \exp_{s,t}(ix,0)&=1+ix=\cos_{s,t}(x,0)+i\sin_{s,t}(x,0),\\
    \exp_{s,t}(z,0)&=1+z=\cos_{s,t}(z,0)+\sin_{s,t}(z,0).
\end{align*}
For $u\neq0$, we use the power series expansion of $\e_{s,t}(ix,u)$ and $\exp_{s,t}(z,-u)$, respectively. Thus, we have that
\begin{align*}
    \exp_{s,t}(ix,u)&=\sum_{n=0}^{\infty}u^{\binom{n}{2}}\frac{(ix)^n}{\brk[c]{n}_{s,t}!}\\
    &=\sum_{n=0}^{\infty}u^{\binom{2n}{2}}(-1)^{n}\frac{x^{2n}}{\brk[c]{2n}_{s,t}!}+i\sum_{n=0}^{\infty}u^{\binom{2n+1}{2}}(-1)^{n}\frac{x^{2n+1}}{\brk[c]{2n+1}_{s,t}!}\\
    &=\cos_{s,t}(x,u)+i\sin_{s,t}(x,u)
\end{align*}
and
\begin{align*}
    \exp_{s,t}(z,-u)&=\sum_{n=0}^{\infty}(-u)^{\binom{n}{2}}\frac{x^{n}}{\brk[c]{n}_{s,t}!}\\
    &=\sum_{n=0}^{\infty}(-1)^{\frac{n(n-1)}{2}}u^{\binom{n}{2}}\frac{x^{n}}{\brk[c]{n}_{s,t}!}\\
    &=\sum_{n=0}^{\infty}(-1)^{n(2n-1)}u^{\binom{2n}{2}}\frac{x^{2n}}{\brk[c]{2n}_{s,t}!}+\sum_{n=0}^{\infty}(-1)^{(2n+1)n}u^{\binom{2n+1}{2}}\frac{x^{2n+1}}{\brk[c]{2n+1}_{s,t}!}\\
    &=\sum_{n=0}^{\infty}(-1)^{n}u^{\binom{2n}{2}}\frac{x^{2n}}{\brk[c]{2n}_{s,t}!}+\sum_{n=0}^{\infty}(-1)^{n}u^{\binom{2n+1}{2}}\frac{x^{2n+1}}{\brk[c]{2n+1}_{s,t}!}\\
    &=\cos_{s,t}(z,u)+\sin_{s,t}(z,u).
\end{align*}
\end{proof}

For all $u\in\C$, we can express the Lucas-Pantograph type sine and cosine functions in terms of the Lucas-Pantograph type exponential function as
\begin{align*}
    \sin_{s,t}(x,u)&=\frac{\exp_{s,t}(ix,u)-\exp_{s,t}(-ix,u)}{2i}
    =\frac{\exp_{s,t}(x,-u)-\exp_{s,t}(-x,-u)}{2},\\
    \cos_{s,t}(x,u)&=\frac{\exp_{s,t}(ix,u)+\exp_{s,t}(-ix,u)}{2}
    =\frac{\exp_{s,t}(x,u)+\exp_{s,t}(-x,u)}{2}.
\end{align*}

Next, we have the Lucas-Pantograph-analogue of the identity
\begin{equation*}
    e^{x+iy}=e^{x}(\cos y+i\sin y).
\end{equation*}
\begin{theorem}
For $x,y,u,v\in\C$
\begin{enumerate}
\item $\exp_{s,t}(x\oplus_{u,v}iy)=\exp_{s,t}(x,u)(\cos_{s,t}(y,v)+i\sin_{s,t}(y,v))$,
\item $\exp_{s,t}(x\oplus_{u,-v}y)=\exp_{s,t}(x,u)(\cos_{s,t}(y,v)+\sin_{s,t}(y,v))$.
\end{enumerate}
\end{theorem}
\begin{proof}
The statements follow from Theorem \ref{prop_exp_complex_neg}.
\end{proof}

\subsection{Parity}

\begin{proposition}\label{prop_sym_trigo}
For all $u\in\C$
\begin{align*}
    \sin_{s,t}(-z,u)&=-\sin_{s,t}(z,u)\text{ and } \cos_{s,t}(-z,u)=\cos_{s,t}(z,u).\\
    \tan_{s,t}(-z,u)&=-\tan_{s,t}(z,u)\text{ and } \cot_{s,t}(-z,u)=\cot_{s,t}(z,u).\\
    \csc_{s,t}(-z,u)&=-\csc_{s,t}(z,u)\text{ and } \sec_{s,t}(-z,u)=\sec_{s,t}(z,u).
\end{align*}
\end{proposition}
\begin{proof}
The functions $\sin_{s,t}(z,u)$ are $\cos_{s,t}(z,u)$ are odd and even functions, respectively. From here follow the first two identities. 
\end{proof}

\subsection{Sum and difference formulas}

In the following definition, we introduce the $(u,v)$-binomial $(s,t)$-trigonometric functions that will be useful for finding $(s,t)$-trigonometric addition and subtraction formulas.
\begin{definition}
For all $u,v,x,y,s,t\in\C$, $s\neq0,t\neq0$, we define the $(u,v)$-binomial trigonometric-Lucas functions as
\begin{align*}
\sin_{s,t}(x\oplus_{u,v}y)&=\sum_{n=0}^{\infty}(-1)^{n}\frac{(x\oplus_{u,v}y)_{s,t}^{(2n+1)}}{\brk[c]{2n+1}_{s,t}!},\\
\cos_{s,t}(x\oplus_{u,v}y)&=\sum_{n=0}^{\infty}(-1)^{n}\frac{(x\oplus_{u,v}y)_{s,t}^{(2n)}}{\brk[c]{2n}_{s,t}!},
\end{align*}
and
\begin{align*}
    \tan_{s,t}(x\oplus_{u,v}y)&=\frac{\sin_{s,t}(x\oplus_{u,v}y)}{\cos_{s,t}(x\oplus_{u,v}y)},\hspace{0.3cm}\cot_{s,t}(x\oplus_{u,v}y)=\frac{\cos_{s,t}(x\oplus_{u,v}y)}{\sin_{s,t}(x\oplus_{u,v}y)}.\\
    \sec_{s,t}(x\oplus_{u,v}y)&=\frac{1}{\cos_{s,t}(x\oplus_{u,v}y)},\hspace{0.3cm}\csc_{s,t}(x\oplus_{u,v}y)=\frac{1}{\sin_{s,t}(x\oplus_{u,v}y)}.
\end{align*}
\end{definition}

\begin{theorem}
For all $x,y,u,v\in\C$
    \begin{equation}
        \e_{s,t}(x\oplus_{-u,-v}y)=\cos_{s,t}(x\ominus_{u,v}y)+\sin_{s,t}(x\oplus_{u,v}y).
    \end{equation}
\end{theorem}
\begin{proof}
From Theorem \ref{theo_bin_neg}
    \begin{align*}
        \e_{s,t}(x\oplus_{-u,-v}y)&=\sum_{n=0}^{\infty}(-1)^{\binom{n}{2}}\frac{(x\oplus_{u,v}(-1)^{1-n}y)_{s,t}^{(n)}}{\brk[c]{n}_{s,t}!}\\
        &=\sum_{n=0}^{\infty}(-1)^{n}\frac{(x\ominus_{u,v}y)_{s,t}^{(2n)}}{\brk[c]{2n}_{s,t}!}+\sum_{n=0}^{\infty}(-1)^{n}\frac{(x\oplus_{u,v}y)_{s,t}^{(2n+1)}}{\brk[c]{2n+1}_{s,t}!}\\
        &=\cos_{s,t}(x\ominus_{u,v}y)+\sin_{s,t}(x\oplus_{u,v}y).
    \end{align*}
\end{proof}

From Definition \ref{def_expbin} we have the following theorem.
\begin{theorem}\label{theo_rep_trigobin}
\begin{equation}
    \sin_{s,t}(x\oplus_{u,v}y)=\frac{\e_{s,t}(ix\oplus_{u,v}iy)-\e_{s,t}((-ix)\oplus_{u,v}(-iy))}{2i}
\end{equation}
and
\begin{equation}
    \cos_{s,t}(x\oplus_{u,v}y)=\frac{\e_{s,t}(ix\oplus_{u,v}iy)+\e_{s,t}((-ix)\oplus_{u,v}(-iy))}{2}.
\end{equation}
\end{theorem}
\begin{proof}
Use Definition \ref{def_expbin} to obtain
\begin{equation*}
    \e_{s,t}(ix\oplus_{u,v}iy)=\cos_{s,t}(x\oplus_{u,v}y)+i\sin_{s,t}(x\oplus_{u,v}y).
\end{equation*}
\end{proof}

\begin{theorem}\label{theo_add_sincos}
For every pair of complex numbers $u,v$, the following are sum and difference identities for the Lucas-Pantograph type sine and cosine functions
\begin{enumerate}
    \item $\sin_{s,t}(x\oplus_{u,v}y)=\sin_{s,t}(x,u)\cos_{s,t}(y,v)+\cos_{s,t}(x,u)\sin_{s,t}(y,v)$,
    \item  $\sin_{s,t}(x\ominus_{u,v}y)=\sin_{s,t}(x,u)\cos_{s,t}(y,v)-\cos_{s,t}(x,u)\sin_{s,t}(y,v)$,
    \item $\cos_{s,t}(x\oplus_{u,v}y)=\cos_{s,t}(x,u)\cos_{s,t}(y,v)-\sin_{s,t}(x,u)\sin_{s,t}(y,v)$,
    \item $\cos_{s,t}(x\ominus_{u,v}y)=\cos_{s,t}(x,u)\cos_{s,t}(y,v)+\sin_{s,t}(x,u)\sin_{s,t}(y,v)$.
\end{enumerate}
\end{theorem}
\begin{proof}
Follows from Theorem \ref{theo_rep_trigobin} and Definition \ref{def_expbin}.
\end{proof}

\begin{theorem}\label{theo_add_tan}
For every pair of complex numbers $u,v$, the following are sum and difference identities for the Lucas-Pantograph type tangent function
\begin{enumerate}
    \item 
    \begin{equation*}
        \tan_{s,t}(x\oplus_{u,v}y)=\frac{\tan_{s,t}(x,u)+\tan_{s,t}(y,v)}{1-\tan_{s,t}(x,u)\tan_{s,t}(y,v)}
    \end{equation*}
    \item 
    \begin{equation*}
        \tan_{s,t}(x\ominus_{u,v}y)=\frac{\tan_{s,t}(x,u)-\tan_{s,t}(y,v)}{1+\tan_{s,t}(x,u)\tan_{s,t}(y,v)}
    \end{equation*}
\end{enumerate}
\end{theorem}
    
\begin{corollary}\label{coro_4}
For all $u,v\in\C$ and replacing $y=-x$ in Theorem \ref{theo_add_sincos}, then
\begin{enumerate}
    \item $\sin_{s,t}(x,u)\cos_{s,t}(x,v)-\cos_{s,t}(x,u)\sin_{s,t}(x,v)=\sin_{s,t}(\overline{0}_{u,v}x)$,
    \item $\vphiS_{s,t}(x)\phiC_{s,t}(x)-\vphiC_{s,t}(x)\phiS_{s,t}(x)=0$.
\end{enumerate}
\end{corollary}

\begin{definition}
Let $u_{1},\ldots,u_{k}$ be non-zero complex numbers and set $\mathbf{u}=(u_{1},\ldots,u_{k})$. Define the $\mathbf{u}$-multinomial sine-Lucas function as
\begin{equation}
    \sin_{s,t}(\overline{k}_{\mathbf{u}}x)=\sum_{n=0}^{\infty}(-1)^n\overline{k}_{\mathbf{u},s,t}^{(2n+1)}\frac{x^{2n+1}}{\brk[c]{2n+1}_{s,t}!}.
\end{equation}
Define the $\mathbf{u}$-multinomial cosine-Lucas function as
\begin{equation}
    \cos_{s,t}(\overline{k}_{\mathbf{u}}x)=\sum_{n=0}^{\infty}(-1)^n\overline{k}_{\mathbf{u},s,t}^{(2n)}\frac{x^{2n}}{\brk[c]{2n}_{s,t}!}.
\end{equation}
Additionally, we define
\begin{equation}
    \tan_{s,t}(\overline{k}_{\mathbf{u}}x)=\frac{\sin_{s,t}(\overline{k}_{\mathbf{u}}x)}{\cos_{s,t}(\overline{k}_{\mathbf{u}}x)},\hspace{0.3cm}\cot_{s,t}(\overline{k}_{\mathbf{u}}x)=\frac{\cos_{s,t}(\overline{k}_{\mathbf{u}}x)}{\sin_{s,t}(\overline{k}_{\mathbf{u}}x)}
\end{equation}
and
\begin{equation}
    \sec_{s,t}(\overline{k}_{\mathbf{u}}x)=\frac{1}{\cos_{s,t}(\overline{k}_{\mathbf{u}}x)},\hspace{0.3cm}\csc_{s,t}(\overline{k}_{\mathbf{u}}x)]=\frac{1}{\sin_{s,t}(\overline{k}_{\mathbf{u}}x)}.
\end{equation}
\end{definition}

\begin{theorem}\label{theo_EL_multi}
For all $x\in\R$,
\begin{equation}
    \e_{s,t}(\overline{k}_{\mathbf{u}}ix)=\cos_{s,t}(\overline{k}_{\mathbf{u}}x)+i\sin_{s,t}(\overline{k}_{\mathbf{u}}x).
\end{equation}
Also, we have that
\begin{equation}
    \sin_{s,t}(\overline{k}_{\mathbf{u}}x)=\frac{\e_{s,t}(\overline{k}_{\mathbf{u}}ix)-\e_{s,t}(-\overline{k}_{\mathbf{u}}ix)}{2i}
\end{equation}
and
\begin{equation}
    \cos_{s,t}(\overline{k}_{\mathbf{u}}x)=\frac{\e_{s,t}(\overline{k}_{\mathbf{u}}ix)+\e_{s,t}(-\overline{k}_{\mathbf{u}}ix)}{2}.
\end{equation}
\end{theorem}

\begin{theorem}\label{theo_add_sincos_n1}
For every $\mathbf{u}=\overbrace{(u,\ldots,u)}^{n}$, the following are sum and difference identities for the $\mathbf{u}$-multinomial sine-Lucas and cosine-Lucas functions
\begin{enumerate}
    \item $\sin_{s,t}(\overline{n}_{\mathbf{u}}x\oplus_{1,u}y)=\sin_{s,t}(\overline{n}_{\mathbf{u}}x)\cos_{s,t}(y,u)+\cos_{s,t}(\overline{n}_{\mathbf{u}}x)\sin_{s,t}(y,u)$,
    \item  $\sin_{s,t}(\overline{n}_{\mathbf{u}}x\ominus_{1,u}y)=\sin_{s,t}(\overline{n}_{\mathbf{u}}x)\cos_{s,t}(y,u)-\cos_{s,t}(\overline{n}_{\mathbf{u}}x)\sin_{s,t}(y,u)$,
    \item $\cos_{s,t}(\overline{n}_{\mathbf{u}}x\oplus_{1,u}y)=\cos_{s,t}(\overline{n}_{\mathbf{u}}x)\cos_{s,t}(y,u)-\sin_{s,t}(\overline{n}_{\mathbf{u}}x)\sin_{s,t}(y,u)$,
    \item $\cos_{s,t}(\overline{n}_{\mathbf{u}}x\ominus_{1,u}y)=\cos_{s,t}(\overline{n}_{\mathbf{u}}x)\cos_{s,t}(y,u)+\sin_{s,t}(\overline{n}_{\mathbf{u}}x)\sin_{s,t}(y,u)$.
\end{enumerate}
\end{theorem}

\begin{theorem}\label{theo_add_sincos_nm}
For every pair $\mathbf{u}=\overbrace{(u,\ldots,u)}^{n}$, $\mathbf{v}=\overbrace{(u,\ldots,u)}^{m}$, the following are sum and difference identities for the $(\mathbf{u},\mathbf{v})$-multinomial sine-Lucas and cosine-Lucas functions
\begin{enumerate}
    \item $\sin_{s,t}(\overline{n}_{\mathbf{u}}x\oplus_{1,1}\overline{m}_{\mathbf{v}}y)=\sin_{s,t}(\overline{n}_{\mathbf{u}}x)\cos_{s,t}(\overline{m}_{\mathbf{v}}y)+\cos_{s,t}(\overline{n}_{\mathbf{u}}x)\sin_{s,t}(\overline{m}_{\mathbf{v}}y)$,
    \item  $\sin_{s,t}(\overline{n}_{\mathbf{u}}x\ominus_{1,1}\overline{m}_{\mathbf{v}}y)=\sin_{s,t}(\overline{n}_{\mathbf{u}}x)\cos_{s,t}(\overline{m}_{\mathbf{v}}y)-\cos_{s,t}(\overline{n}_{\mathbf{u}}x)\sin_{s,t}(\overline{m}_{\mathbf{v}}y)$,
    \item $\cos_{s,t}(\overline{n}_{\mathbf{u}}x\oplus_{1,1}\overline{m}_{\mathbf{v}}y)=\cos_{s,t}(\overline{n}_{\mathbf{u}}x)\cos_{s,t}(\overline{m}_{\mathbf{v}}y)-\sin_{s,t}(\overline{n}_{\mathbf{u}}x)\sin_{s,t}(\overline{m}_{\mathbf{v}}y)$,
    \item $\cos_{s,t}(\overline{n}_{\mathbf{u}}\ominus_{1,1}\overline{m}_{\mathbf{v}}y)=\cos_{s,t}(\overline{n}_{\mathbf{u}}x)\cos_{s,t}(\overline{m}_{\mathbf{v}}y)+\sin_{s,t}(\overline{n}_{\mathbf{u}}x)\sin_{s,t}(\overline{m}_{\mathbf{v}}y)$.
\end{enumerate}
\end{theorem}

\subsection{Pythagorean identities}

\begin{corollary}\label{coro_pytha}
For all $u,v\in\C$ and replacing $y=-x$ in Theorem \ref{theo_add_sincos}, then the Lucas-Pantograph-Pythagorean identities are
\begin{enumerate}
    \item $\sin_{s,t}(x,u)\sin_{s,t}(x,v)+\cos_{s,t}(x,u)\cos_{s,t}(x,v)=\cos_{s,t}(\overline{0}_{u,v}x)$,
    \item $\sin_{s,t}^{2}(x,u)+\cos_{s,t}^{2}(x,u)=\cos_{s,t}(\overline{0}_{u,u}x)$.
    \item $\vphiS_{s,t}(x)\phiS_{s,t}(x)+\vphiC_{s,t}(x)\phiC_{s,t}(x)=1$.
\end{enumerate}
\end{corollary}

Now, define the functions
\begin{align}\label{eqn_trigo_wt1}
    \sinwt_{s,t}(x,u)&=\frac{\sin_{s,t}(x,u)}{\sqrt{\cos_{s,t}(\overline{0}_{u,u}x)}},\hspace{0.5cm}\coswt_{s,t}(x,u)=\frac{\cos_{s,t}(x,u)}{\sqrt{\cos_{s,t}(\overline{0}_{u,u}x)}},
\end{align}
\begin{equation}\label{eqn_trigo_wt2}
    \tanwt_{s,t}(x,u)=\tan_{s,t}(x,u),\hspace{0.5cm}\cotwt_{s,t}(x,u)=\cot_{s,t}(x,u),
\end{equation}
and
\begin{align}\label{eqn_trigo_wt3}
    \secwt_{s,t}(x,u)=\frac{\sqrt{\cos_{s,t}(\overline{0}_{u,u}x)}}{\cos_{s,t}(x,u)},\hspace{0.5cm}\cscwt_{s,t}(x,u)=\frac{\sqrt{\cos_{s,t}(\overline{0}_{u,u}x)}}{\sin_{s,t}(x,u)}.
\end{align}
Therefore, from Corollary \ref{coro_pytha} and Eqs. (\ref{eqn_trigo_wt1}), (\ref{eqn_trigo_wt2}), and (\ref{eqn_trigo_wt3})
\begin{align}
    \sinwt_{s,t}(x,u)^{2}+\coswt_{s,t}(x,u)^{2}&=1,\label{eqn_pytha1}\\
    \tanwt_{s,t}^2(x,u)+1&=\secwt_{s,t}^2(x,u),\label{eqn_pytha2}\\
    1+\cotwt_{s,t}^2(x,u)&=\cscwt_{s,t}^2(x,u).\label{eqn_pytha3}
\end{align}

\subsection{Double-angle identities}

\begin{corollary}\label{coro_double-angle}
For all $u,v\in\C$ and replacing $y=x$ in Theorems \ref{theo_add_sincos} and Theorem \ref{theo_add_tan}, then the Lucas-Pantograph type double-angle formulae are
\begin{enumerate}
    \item $\sin_{s,t}(\overline{2}_{u,v}x)=\sin_{s,t}(x,u)\cos_{s,t}(x,v)+\cos_{s,t}(x,u)\sin_{s,t}(x,v)$.
    \item $\sin_{s,t}(\overline{2}_{u,u}x)=2\sin_{s,t}(x,u)\cos_{s,t}(x,u)$.
    \item $\cos_{s,t}(\overline{2}_{u,v}x)=\cos_{s,t}(x,u)\cos_{s,t}(x,v)-\sin_{s,t}(x,u)\sin_{s,t}(x,v)$.
    \item $\cos_{s,t}(\overline{2}_{u,u}x)=\cos_{s,t}^{2}(x,u)-\sin_{s,t}^{2}(x,u)$.
    \item $\tan_{s,t}(\overline{2}_{u,v}x)=\frac{\tan_{s,t}(x,u)+\tan_{s,t}(x,v)}{1-\tan_{s,t}(x,u)\tan_{s,t}(x,v)}$.
    \item $\tan_{s,t}(\overline{2}_{u,u}x)=\frac{2\tan_{s,t}(x,u)}{1-\tan_{s,t}^2(x,u)}$.
\end{enumerate}
\end{corollary}

\subsection{Special values in $x=\pi_{u}$}

\begin{theorem}
Suppose that $x=\pi_{u}$ such that $\sin_{s,t}(\pi_{u},u)=0$ and $\cos_{s,t}(\overline{0}_{u,u}\pi_{u})\neq0$. Then
    \begin{enumerate}
        \item $\sin_{s,t}(\overline{n}_{u}\pi_{u})=0$.
        \item $\cos_{s,t}(\overline{n}_{u}\pi_{u})=\pm\cos_{s,t}^{n/2}(\overline{0}_{u,u}\pi_{u})$, $n\geq1$.
        \item $\tan_{s,t}(\overline{n}_{u}\pi_{u})=0$.
    \end{enumerate}
\end{theorem}
\begin{proof}
The proof is by induction on $n$. Using Corollary \ref{coro_double-angle}
\begin{align*}
    \sin_{s,t}(\overline{2}_{u,u}\pi_{u})&=2\sin_{s,t}(\pi_{u},u)\cos_{s,t}(\pi_{u},u)=0.
\end{align*}
Suppose that $\sin_{s,t}(\overline{n}_{\mathbf{u}}\pi_{u})=0$, with $\mathbf{u}=(u,\ldots,u)$. Then from the Theorem \ref{theo_add_sincos_n1},
\begin{align*}
    \sin_{s,t}(\overline{n+1}_{\mathbf{u},u}\pi_{u})&=\sin_{s,t}((\overline{n}_{\mathbf{u}}\oplus_{1,u}1)\pi_{u})\\
    &=\sin_{s,t}(\overline{n}_{\mathbf{u}}\pi_{u})\cos_{s,t}(\pi_{u},u)+\cos_{s,t}(\overline{n}_{\mathbf{u}}\pi_{u})\sin_{s,t}(\pi_{u},u)\\
    &=0.
\end{align*}
This proves statement 1. From Eqs. (\ref{eqn_pytha1}) and (\ref{eqn_trigo_wt1}), it is clear that
\begin{equation}
    \cos_{s,t}(\pi_{u},u)=\pm\sqrt{\cos_{s,t}(\overline{0}_{u,u}\pi_{u})}\neq0.
\end{equation}
From Corollary \ref{coro_double-angle},
\begin{equation}
    \cos_{s,t}(\overline{2}_{u,u}\pi_{u})=\cos_{s,t}^2(\pi_{u},u)=\cos_{s,t}(\overline{0}_{u,u}\pi_{u}).
\end{equation}
By using the Theorem \ref{theo_add_sincos_n1}, we have that
\begin{align*}
    \cos_{s,t}(\overline{n+1}_{\mathbf{u},u}\pi_{u})&=\cos_{s,t}(\overline{n}_{\mathbf{u}}\pi_{u})\cos_{s,t}(\pi_{u},u)-\sin_{s,t}(\overline{n}_{\mathbf{u}}\pi_{u})\sin_{s,t}(\pi_{u},u)\\
    &=\pm\cos_{s,t}^{n/2}(\overline{0}_{u,u}\pi_{u})\sqrt{\cos_{s,t}(\overline{0}_{u,u}\pi_{u})}\\
    &=\pm\cos_{s,t}^{(n+1)/2}(\overline{0}_{u,u}\pi_{u}).
\end{align*}
This proves statement 2. Finally,
\begin{equation}
    \tan_{s,t}(\overline{n}_{\mathbf{u}}\pi_{u})=\frac{\sin_{s,t}(\overline{n}_{\mathbf{u}}\pi_{u})}{\cos_{s,t}(\overline{n}_{\mathbf{u}}\pi_{u})}=0.
\end{equation}
Thus the statement 3 is proved.
\end{proof}


\subsection{Periodic functions}

\begin{corollary}
    \begin{enumerate}
        \item $\sin_{s,t}((\overline{n}_{\mathbf{u}}\pi_{u}\oplus_{1,v}x))=\cos_{s,t}^{n/2}(\overline{0}_{u,u}\pi_{u})\sin_{s,t}(x,v)$.
        \item $\cos_{s,t}((\overline{n}_{\mathbf{u}}\pi_{u}\oplus_{1,v}x))=\cos_{s,t}^{n/2}(\overline{0}_{u,u}\pi_{u})\cos_{s,t}(x,v)$.
        \item $\tan_{s,t}((\overline{n}_{\mathbf{u}}\pi_{u}\oplus_{1,v}x))=\tan_{s,t}(x,v)$.
        \item $\cot_{s,t}((\overline{n}_{\mathbf{u}}\pi_{u}\oplus_{1,v}x))=\cot_{s,t}(x,v)$.
    \end{enumerate}
\end{corollary}

\section{Hyperbolic-Lucas functions}

\begin{definition}
For all $s,t,u\in\C$, $s\neq0,t\neq0$, the Lucas-Pantograph type hyperbolic functions are defined as
\begin{align*}
    \sinh_{s,t}(z,u)&=
    \begin{cases}
        \sum_{n=0}^{\infty}u^{\binom{2n+1}{2}}\frac{z^{2n+1}}{\brk[c]{2n+1}_{s,t}!},&\text{ if }u\neq0;\\
        z,&\text{ if }u=0,
    \end{cases}\\
    \cosh_{s,t}(z,u)&=
    \begin{cases}
        \sum_{n=0}^{\infty}u^{\binom{2n}{2}}\frac{z^{2n}}{\brk[c]{2n}_{s,t}!},&\text{ if }u\neq0;\\
        1,&\text{ if }u=0,
    \end{cases}
\end{align*}
and
\begin{align*}
    \tanh_{s,t}(z,u)&=\frac{\sinh_{s,t}(z,u)}{\cosh_{s,t}(z,u)},\hspace{0.3cm}\coth_{s,t}(z,u)=\frac{\cosh_{s,t}(z,u)}{\sinh_{s,t}(z,u)},\\
    \sech_{s,t}(z,u)&=\frac{\sinh_{s,t}(z,u)}{\cosh_{s,t}(z,u)},\hspace{0.3cm}\csch_{s,t}(z,u)=\frac{\cosh_{s,t}(z,u)}{\sinh_{s,t}(z,u)}
\end{align*}
\end{definition}

\begin{theorem}\label{theo_trigo-hyp}
    \begin{align*}
        \sin_{s,t}(ix,u)&=i\sinh_{s,t}(x,u),\hspace{0.3cm}\sin_{s,t}(x,-u)=\sinh_{s,t}(x,u).\\
        \cos_{s,t}(ix,u)&=\cosh_{s,t}(x,u),\hspace{0.3cm}\cos_{s,t}(x,-u)=\cosh_{s,t}(x,u).\\
        \tan_{s,t}(ix,u)&=i\tanh_{s,t}(x,u),\hspace{0.3cm}\tan_{s,t}(x,-u)=\tanh_{s,t}(x,u).
    \end{align*}
\end{theorem}

\begin{definition}
For all $u,v,x,y,s,t\in\C$, $s\neq0,t\neq0$, we define the $(u,v)$-binomial hyperbolic-Lucas functions as
\begin{align*}
\sinh_{s,t}(x\oplus_{u,v}y)&=\sum_{n=0}^{\infty}\frac{(x\oplus_{u,v}y)_{s,t}^{(2n+1)}}{\brk[c]{2n+1}_{s,t}!},\\
\cosh_{s,t}(x\oplus_{u,v}y)&=\sum_{n=0}^{\infty}\frac{(x\oplus_{u,v}y)_{s,t}^{(2n)}}{\brk[c]{2n}_{s,t}!},
\end{align*}
and
\begin{align*}
    \tanh_{s,t}(x\oplus_{u,v}y)&=\frac{\sinh_{s,t}(x\oplus_{u,v}y)}{\cosh_{s,t}(x\oplus_{u,v}y)},\hspace{0.3cm}\coth_{s,t}(x\oplus_{u,v}y)=\frac{\cosh_{s,t}(x\oplus_{u,v}y)}{\sinh_{s,t}(x\oplus_{u,v}y)}.\\
    \sech_{s,t}(x\oplus_{u,v}y)&=\frac{1}{\cosh_{s,t}(x\oplus_{u,v}y)},\hspace{0.3cm}\csch_{s,t}(x\oplus_{u,v}y)=\frac{1}{\sinh_{s,t}(x\oplus_{u,v}y)}.
\end{align*}
\end{definition}

\begin{theorem}\label{theo_bin_trigo-hyp}
    \begin{align*}
        \sin_{s,t}(x\oplus_{-u,-v}y)&=\sinh_{s,t}(x\oplus_{u,v}y).\\
        \cos_{s,t}(x\oplus_{-u,-v}y)&=\cosh_{s,t}(x\ominus_{u,v}y).
    \end{align*}
\end{theorem}

\begin{theorem}\label{theo_add_hyp}
For every pair of complex numbers $u,v$, the following are negative sum and difference identities for the Lucas-Pantograph type sine and cosine functions
\begin{enumerate}
    \item $\sinh_{s,t}(x\oplus_{u,v}y)=\sinh_{s,t}(x,u)\cosh_{s,t}(y,v)+\cosh_{s,t}(x,u)\sinh_{s,t}(y,v)$,
    \item  $\sinh_{s,t}(x\ominus_{u,v}y)=\sinh_{s,t}(x,u)\cosh_{s,t}(y,v)-\cosh_{s,t}(x,u)\sinh_{s,t}(y,v)$,
    \item $\cosh_{s,t}(x\oplus_{u,v}y)=\cosh_{s,t}(x,u)\cosh_{s,t}(y,v)+\sinh_{s,t}(x,u)\sinh_{s,t}(y,v)$,
    \item $\cosh_{s,t}(x\ominus_{u,v}y)=\cosh_{s,t}(x,u)\cosh_{s,t}(y,v)-\sinh_{s,t}(x,u)\sinh_{s,t}(y,v)$.
\end{enumerate}
\end{theorem}
\begin{proof}
Follows from Theorems \ref{theo_bin_trigo-hyp} and \ref{theo_add_sincos}, and then Theorem \ref{theo_trigo-hyp}.
\end{proof}

\begin{theorem}
For every pair of complex numbers $u,v$, the following are sum and difference identities for the Lucas-Pantograph type hyperbolic tangent function
\begin{enumerate}
    \item 
    \begin{equation*}
        \tanh_{s,t}(x\oplus_{u,v}y)=\frac{\tanh_{s,t}(x,u)+\tanh_{s,t}(y,v)}{1+\tanh_{s,t}(x,u)\tanh_{s,t}(y,v)}
    \end{equation*}
    \item 
    \begin{equation*}
        \tanh_{s,t}(x\ominus_{u,v}y)=\frac{\tanh_{s,t}(x,u)-\tanh_{s,t}(y,v)}{1-\tanh_{s,t}(x,u)\tanh_{s,t}(y,v)}
    \end{equation*}
\end{enumerate}
\end{theorem}

\end{document}